\documentclass[revision]{FPSAC2018}

\newtheorem{theorem}{Theorem}

\newtheorem{example}[theorem]{Example}
\newtheorem{remark}[theorem]{Remark}
\usepackage{lipsum}
\usepackage{ytableau}
\usepackage{mathtools}

\newcommand{\kboolean}[1]{X^{(#1)}} 
\newcommand{\ssyt}[1]{\mathrm{SSYT}(#1)} 
\newcommand{\cont}[1]{\mathrm{cont}(#1)} 
\newcommand{\schurfunc}[1]{\mathbb{S}^{#1}} 
\newcommand{\Exterior}{\mathchoice{{\textstyle\bigwedge}}%
    {{\bigwedge}}%
    {{\textstyle\wedge}}%
    {{\scriptstyle\wedge}}}
\title{Boolean product polynomials and Schur-positivity}

\author{Louis
  J.~Billera\thanks{\href{mailto:billera@math.cornell.edu}{billera@math.cornell.edu}. Supported
    in part by the Simons Foundation.}\addressmark{1},  Sara
  C.~Billey\thanks{\href{mailto:billey@math.washington.edu}{billey@math.washington.edu}. Partially
    supported by grant DMS-1101017 from the NSF.}\addressmark{2} \and Vasu Tewari\addressmark{2}}

\address{\addressmark{1}Department of Mathematics, Cornell University, Ithaca, NY 14853 \\ \addressmark{2}Department of Mathematics, University of Washington, Seattle, WA 98195}

\received{\today}

\revised{}

\abstract{
We study a family of symmetric polynomials that we refer to as the Boolean product polynomials.
The motivation for studying these polynomials stems from the computation of the characteristic polynomial of the real matroid spanned by the nonzero vectors in $\mathbb{R}^n$ all of whose coordinates are either $0$ or $1$. To this end, one approach is to compute the zeros of the Boolean product polynomials over finite fields.
The zero loci of these polynomials cut out  hyperplane arrangements known as resonance arrangements, which show up in the context of double Hurwitz polynomials.
By relating the Boolean product polynomials to certain total Chern classes
of vector bundles, we establish their Schur-positivity by appealing to
a result of Pragacz relying on earlier work on numerical positivity by Fulton-Lazarsfeld. Subsequently, we study a two-alphabet version of these polynomials from the viewpoint of Schur-positivity. As a special case of these polynomials, we recover  symmetric functions first studied by D\'esarm\'enien and Wachs in the context of descents in derangements.
}


\usepackage[backend=bibtex]{biblatex}
\begin{document}

\maketitle

\section{Introduction}
Consider the polynomial ring $\mathbb{C}[x_1,\ldots,x_n]$ in $n$ commuting indeterminates $x_1,\ldots, x_n$.
Let $X=\{x_1,\ldots,x_n\}$ and $[n]=\{1,\ldots,n\}$.
For a nonempty subset $S$ of $[n]$, define the linear form $X_S$ in
$\mathbb{C}[x_1,\ldots,x_n]$ by the sum 
\[
X_S\coloneqq\sum_{i\in S}x_i.
\]
For integers $k,n$ satisfying $1\leq k\leq n$, define homogeneous
polynomials $B_{n,k}(X)$ of degree $\binom{n}{k}$ and $B_{n}(X)$ of
degree $2^n - 1$ by the products 
\begin{equation}\label{eq:b.nk}
  B_{n,k}(X)\coloneqq\prod_{S\subseteq [n], |S|=k} X_S \hspace{.1in}
  \text{ and } \hspace{.1in} 
  B_n(X) \coloneqq\prod_{k=1}^n B_{n,k}(X).
\end{equation}
For brevity, let $B_{n,k}\coloneqq B_{n,k}(X)$ and
$B_n\coloneqq B_n(X)$ when the alphabet is understood.  We refer to
$B_{n,k}$ as the \emph{$(n,k)$-th Boolean product polynomial}
and to $B_n$ as the \emph{$n$-th total Boolean product polynomial}.

Observe that the polynomials $B_{n}(X)$ and $B_{n,k}(X)$ are symmetric
under permutations of the variables.  In fact, one might consider the
Boolean product polynomials to be among the most ``natural symmetric
polynomials'' along with elementary, homogeneous, and monomial
symmetric polynomials, and Stanley's chromatic symmetric polynomials.
The main focus of this article is the Schur-positivity of all Boolean product polynomials.

\begin{theorem}\label{thm:Billey-Billera} For any  positive integers
  $k\leq n$, the $(n,k)$-th Boolean product polynomial $B_{n,k}(X)$ is
  Schur-positive. That is, there exist nonnegative integers
  $\kappa_{\lambda}^{(n,k)}$ such that
 \[
B_{n,k}(X)=\sum_{\lambda} \kappa_{\lambda}^{(n,k)}s_{\lambda}(X).
\]
Furthermore, the $n$-th total Boolean product polynomial is
  Schur-positive.
\end{theorem}

Our motivation for studying the Boolean product polynomials stems from three
long-standing open problems.  One  comes from matroid
theory/hyperplane arrangements, one comes from economics/game
theory/physics, and one comes from Hadamard's maximal determinant problem (see Section~\ref{s:remarks}).

The first interesting open problem is to find the characteristic
polynomial $\chi_{n}(t)$ for the real matroid $M_n$ spanned by 0--1
vectors in $\mathbb{R}^n$.  The same polynomial $\chi_{n}(t)$ is the
characteristic polynomial of the hyperplane arrangement corresponding
to hyperplanes given by the vanishing of $X_S$ for all nonempty
$S \subseteq [n]$.  This arrangement has also appeared in the work of
\cite{CJM11}, where it is called the \emph{resonance arrangement}. One
of the main results in \cite{CJM11} is that the regions of the
resonance arrangement are in fact the chambers of polynomiality of the
genus $g$ double Hurwitz numbers.  One approach to computing
$\chi_{n}(t)$ is the \emph{finite field method} due to Athanasiadis
\cite{Athanasiadis96}, which asserts that, for large enough primes $p$,
$\chi_{n}(p)$ is the number of points in $\mathbb{F}_p^n$ in the
complement of the resonance arrangement.  The finite field method is
also described nicely in \cite{Stanley07} and \cite[Section
3.11.4]{Stanley-EC1}.  This approach was used in \cite[Lemma
5.3]{Terao11} to compute $\chi_{n}$ for small $n$.

As a variety, the resonance  arrangement is the zero locus
of the total Boolean product polynomial $B_{n}(X)$ as defined in
\eqref{eq:b.nk}.  In addition, the zero locus of $B_{n,k}(X)$ in
$\mathbb{R}^n$ is a central hyperplane arrangement in $\mathbb{R}^n$.
Klivans and Reiner \cite{Klivans-Reiner06} study the zonotope dual to
this subarrangement of the resonance  arrangement in the
context of degree sequences of hypergraphs.
 This zonotope, called the
 polytope of degree sequences in \cite{Klivans-Reiner06} is the
 Minkowski sum of the line segments $[0,\mathbf{e}_{S}]$ where $S$
 ranges over all $k$-subsets of $[n]$,\
 $\mathbf{e}_{S}\coloneqq \sum_{i\in S}\mathbf{e}_{i}$,\ and
 $\mathbf{e}_{i}$ is the i$^{th}$ unit vector in $\mathbb{R}^{n}$.

Our second motivating problem has roots in the work of Shapley
\cite{Shapley67} in his study of \emph{economic equilibria}, \emph{i.e.}
the core of an $n$-person cooperative game.  We start by considering
collections of subsets of the finite set $[n]$ as sets of vertices of
the $n$-cube $[0,1]^{n}$.
Let $2^{[n]}$ denote the set of all subsets of $[n]$, and let
$2^{2^{[n]}}$ denote the Boolean algebra on $2^{[n]}$.  A collection
$\mathcal{C} \subseteq 2^{[n]}$ is \emph{balanced} if the convex hull of
the set $\{\mathbf{e}_{S} : S \in \mathcal{C} \}$ meets the main
diagonal in $[0,1]^{n}$, {\it i.e.}, the line between
$\mathbf{e}_{\emptyset}$ and $\mathbf{e}_{[n]}$.  A collection is
\emph{unbalanced} otherwise.  The set of unbalanced collections is an
order ideal in $2^{2^{[n]}}$, while the complementary set of balanced
collections is an order filter.  Thus, it makes sense to consider
\emph{minimal balanced} collections and \emph{maximal unbalanced}
collections.  The former were first considered by Shapley, while the
latter have arisen more recently in two independent studies by
Billera-Moore-Moraites-Wang-Williams \cite{BMMWW12} and Bj\"orner
\cite{Bjorner15}.

We are interested in the enumerative problem of counting the maximal
unbalanced collections for a given $n$.  These collections of subsets
of $[n+1]$ are in bijection with the regions in the 
resonance arrangement \cite{BMMWW12}.  In addition, the regions of the
resonance arrangement are said to count so-called ``generalized
retarded functions'' of quantum field theory \cite{Evans95}, while in
\cite{Terao11}, where the arrangement is called the \emph{all-subsets
  arrangement}, its regions are shown to correspond to certain
preference rankings of interest in psychology and economics.
From work of Zuev \cite{Zuev89}, it is known that the number of maximal
unbalanced collections for a given $n$ is asymptotically on the order
of $2^{n^{2}}$, while specific upper and lower bounds were derived in
\cite{BMMWW12}.  One way of obtaining the number exactly is to compute
the characteristic polynomial $\chi_{n}(t)$ of the resonance
arrangement as described above and to apply the theorem of Zaslavsky
relating regions of hyperplane arrangements to the characteristic
polynomial \cite{Zaslavsky75}, see also \cite[Thm
3.11.7]{Stanley-EC1}.  See \cite[A034997]{oeis}
for additional known results on this integer sequence.



The outline of this extended abstract is as follows.  In
Section~\ref{s:notation}, we review our notation and key theorems by
Lascoux and Pragacz.  In Section~\ref{s:main}, we prove our main
results.  As a consequence, we consider the special case of
$B_{n,n-1}(X)$ where we can give the explicit Schur expansion using
work of D\'esarm\'enien and Wachs on descent sets of derangements.  We
also generalize the Boolean product polynomials to multiple alphabets, and
give additional positivity results.  In Section~\ref{s:remarks}, we
state some additional open problems related to Boolean product
polynomials.

\smallskip

\noindent\textbf{Acknowledgments.} We thank Patricia Hersh, Steve Mitchell, and Jair
Taylor for helpful discussions.  We thank BIRS for the opportunity to
begin this collaboration at the Algebraic Combinatorics Workshop in
August 2015.

\section{Notation and Background}\label{s:notation}
Throughout this extended abstract, we fix a positive integer $n$ and
an alphabet $X=\{x_1,\ldots, x_n\}$. Denote the symmetric group on $n$
letters by $\mathfrak{S}_n$.  We refer the reader to Fulton-Harris
\cite{FultonHarris91}, Macdonald \cite{Macdonald95} or Stanley
\cite[Chapter 7]{Stanley-EC2} for a detailed treatment of the
combinatorics of symmetric polynomials and its relation to the
representation theory of both the symmetric group and the general
linear group.

\subsection{Partitions, tableaux and symmetric polynomials}
A \emph{partition} $\lambda$ of a positive integer $m$ is a finite ordered list of positive integers $(\lambda_1,\ldots,\lambda_k)$ such that $\sum_{i=1}^k\lambda_i=m$. 
We call the $\lambda_i$ the \emph{parts} of $\lambda$ and denote the number of parts by $\ell(\lambda)$. We denote the sum of the parts of $\lambda$ by $|\lambda|$.
If $\lambda$ is a partition of $m$, we denote this by $\lambda\vdash m$. 
Pictorially we depict $\lambda=(\lambda_1,\cdots,\lambda_k)\vdash m$
via its Young diagram drawn in French notation, which is a left-justified array of $m$ boxes with $\lambda_i$ boxes in row $i$ from the bottom.  
Finally, we denote the unique partition of $0$ by $\varnothing$. 

A \emph{semistandard Young tableau} $T$ of shape $\lambda$ is a filling of the boxes of its Young diagram with positive integers such that the entries in each row increase weakly when read from left to right, whereas the entries in each column increase strictly when read from bottom to top.
For any positive integer $m$ and partition $\lambda$, we denote by $\ssyt{\lambda,m}$ the set of semistandard Young tableaux $T$ of shape $\lambda$ satisfying the condition that their entries do not exceed $m$.
A semistandard Young tableau $T$ of shape $\lambda$ with distinct
entries drawn from the set $[|\lambda|]$ is said to be \emph{standard}. 
An entry $i$ in a standard Young tableau (abbreviated SYT) $T$ is a \emph{descent} if $i+1$ belongs to a row strictly above that occupied by $i$. Otherwise, it is an \emph{ascent}.

The symmetric group $\mathfrak{S}_n$ acts on
$\mathbb{C}[x_1,\ldots,x_n]$ by permuting variables. The resulting
ring of invariants, denoted by $\Lambda_n$, is the well-known ring of
\emph{symmetric polynomials} in $n$ variables. It is a polynomial
algebra generated by the $e_p(X)$ for
$1\leq p\leq n$ defined by 
\begin{align}\label{eqn:definition of elementary}
e_p(X)=\sum_{1\leq j_1<\cdots <j_p\leq n} x_{j_1}\cdots x_{j_p}.
\end{align}
We refer to $e_p(X)$ as the \emph{$p$-th elementary symmetric polynomial}. Given a partition $\lambda=(\lambda_1,\ldots, \lambda_k)$, define $e_{\lambda}(X)$ multiplicatively by setting $e_{\lambda}(X)=e_{\lambda_1}(X)\cdots e_{\lambda_k}(X)$. Furthermore, set $e_{\varnothing}(X)=1$.
The ring of symmetric polynomials is a graded ring with the grading given by setting $\deg(e_p(X))=p$. The $d$-th degree graded piece, denote by $\Lambda_n^{d}$, is the $\mathbb{C}$-linear span of the $e_{\lambda}(X)$ where $\lambda\vdash d$ and $\ell(\lambda)\leq n$.
The ring of symmetric polynomials is endowed with a distinguished involution $\omega$ that maps $e_{\lambda}(X)$ to the \emph{complete homogeneous symmetric polynomial} $h_{\lambda}(X)$.

From \eqref{eqn:definition of elementary}, it is clear how to define $e_p(Y)$ for any finite alphabet $Y$. 
Note further that $e_p(Y)$ is $0$ if $p>|Y|$.
Given a positive integer $1\leq k\leq n$, define the following new alphabet 
$$\kboolean{k}\coloneqq \left\lbrace\displaystyle X_S= \sum_{i\in S}x_i :  S\subseteq [n], |S|=k\right\rbrace.$$
The $(n,k)$-th Boolean product polynomial can alternatively be written as $e_{\binom{n}{k}}(\kboolean{k})$.

The most important linear basis of $\Lambda_n$ is given by the Schur functions $s_{\lambda}(X)$ for all partitions $\lambda$. Consider $T\in \ssyt{\lambda,n}$ and let $\cont{T}=(\alpha_1,\ldots,\alpha_n)$ be the ordered sequence of nonnegative integers where $\alpha_i$ is the number of instances of $i$ in $T$, for $1\leq i\leq n$.
Let $X^{\cont{T}}\coloneqq \prod_{i=1}^n x_i^{\alpha_i}$.
The Schur function $s_{\lambda}(X)$ is defined as follows.
\begin{align}
s_{\lambda}(X)=\sum_{T\in \ssyt{\lambda,n}} X^{\cont{T}}.
\end{align}
Since we also work with alphabets other than $X$, we remark here that
to define $s_{\lambda}(Y)$ for any finite alphabet $Y$, the sole
change required, other than changing $X$ to $Y$,  is to replace $n$ by the cardinality of $Y$ throughout. 


\begin{example}\label{ex:(3,2)-boolean sum}
If $n=3$ and $k=2$, the alphabet $\kboolean{2}= \{x_1+x_2,x_1+x_3,x_2+x_3\}$.  We have
\begin{align*}
e_1(\kboolean{2})=&(x_1+x_2)+(x_1+x_3)+(x_2+x_3)= 2s_{(1)}(x_1,x_2,x_3),\\
e_2(\kboolean{2})= &(x_1+x_2)(x_1+x_3)+(x_1+x_2)(x_2+x_3)+(x_1+x_3)(x_2+x_3)\\ =&2s_{(11)}(x_1,x_2,x_3)+s_{(2)}(x_1,x_2,x_3),\\
e_3(\kboolean{2})= &(x_1+x_2)(x_1+x_3)(x_2+x_3)=s_{(21)}(x_1,x_2,x_3).
\end{align*}
As a more involved example, consider $n=5$, $k=3$ and $X=\{x_1,\ldots,x_5\}$.
The reader may verify that $e_{10}(X^{(3)})$ equals the expression below, where the commas and parenthesis in our notation for partitions and the alphabet $X$ have all been omitted:
\begin{align*}
 &6s_{32221} + 9s_{3 3 2 1 1} + 3s_{3 3 2 2} + 3s_{3 3 3 1} + 9s_{4 2 2 1 1} + 3s_{4 2 2 2} + 6s_{4 3 1 1 1} + 9s_{4 3 2 1} + 3s_{4 3 3} + 3s_{4 4 1 1} \\ &+ 3s_{4 4 2} + 4s_{5 2 1 1 1} + 4s_{5 2 2 1} + 4s_{5 3 1 1} + 4s_{5 3 2} + 2s_{5 4 1} + s_{6 1 1 1 1} + s_{6 2 1 1} + s_{6 2 2} + s_{6 3 1}.
\end{align*}
\end{example}
Example~\ref{ex:(3,2)-boolean sum} suggests that the $e_{p}(\kboolean{k})$ expand positively in terms of Schur functions.
This is indeed true and to establish this fact, we need a geometric perspective on obtaining the alphabet $\kboolean{k}$ starting from $X$.

\subsection{Schur functors and Chern classes of vector bundles}
We briefly discuss  some representation theory of the general linear group $\mathrm{GL}_n(\mathbb{C})$ and the symmetric group  $\mathfrak{S}_n$. 
The reader is referred to \cite[Lecture 6]{FultonHarris91} for more details.
Consider a  vector space $V$ of dimension $n$ over $\mathbb{C}$.
 We denote the irreducible polynomial representation of $\mathrm{GL}_{n}(\mathbb{C})$ corresponding to $\lambda\vdash m$ by $\schurfunc{\lambda}(V)$, obtained by acting with the Young symmetrizer $c_{\lambda}$ on $V^{\otimes m}$.  We assume here that $\ell(\lambda)\leq n$. The association $V\mapsto \schurfunc{\lambda}(V)$ is a functor in the category of finite dimensional vector spaces and is called the \emph{Schur functor}.
In particular, $\schurfunc{(1^k)}(V)$ corresponds to the exterior power $\Exterior^{k}V$, whereas $\schurfunc{(k)}(V)$ corresponds to the symmetric power $\mathrm{Sym}^k V$.
The connection to the ring of symmetric polynomials is made explicit by the \emph{character} map  $\mathsf{Ch}$ defined by 
\[
\mathsf{Ch}(\schurfunc{\lambda}(V))=s_{\lambda}(X).
\]
Just as partitions index the irreducible polynomial representations of $\mathrm{GL}_n(\mathbb{C})$, they index the irreducible representations of  $\mathfrak{S}_n$.
The link to Schur polynomials is made manifest by the map $\mathsf{Frob}$, referred to as the \emph{Frobenius characteristic}, that sends the irreducible representation of $\mathfrak{S}_n$ indexed by $\lambda\vdash n$ to $s_{\lambda}(X)$.

We turn our attention to Chern classes of vector bundles over a smooth projective variety $V$. The reader is referred to \cite{Fulton84} for further details. We  merely collect facts that  allow us to cast the question of the Schur-positivity of the Boolean product polynomials as one involving Chern roots.
Let $\mathcal{E}$ be a vector bundle of rank $r$ over  $V$. The \emph{total Chern class} $c(\mathcal{E})$ is the sum of the individual Chern classes
\[
c(\mathcal{E})=1+c_1(\mathcal{E})+\cdots +c_r(\mathcal{E}).
\]
Note that $c_i(\mathcal{E})=0$ for all $i>r$ \cite[Theorem 3.2a]{Fulton84}. 
If one assumes temporarily that $\mathcal{E}$ is the direct sum of line bundles $\mathcal{L}_1,\ldots, \mathcal{L}_r$, then the Whitney-sum property \cite[Theorem 3.2e]{Fulton84} implies that 
\[
c(\mathcal{E})=\prod_{i=1}^{r}(1+c_1(\mathcal{L}_i)).
\]
If $E$ is not a direct sum, the splitting principle \cite[Remark 3.2.3]{Fulton84} says that by constructing an appropriate filtration of $\mathcal{E}$ where the successive quotients are line bundles, one may still factor the total Chern class of $\mathcal{E}$ formally as 
$
c(\mathcal{E})=\prod_{i=1}^{r}(1+\alpha_i).
$
The $\alpha_i$ for $1\leq i\leq r$ are said to be the \emph{Chern roots} of $\mathcal{E}$.
We treat  Chern roots as formal variables. 
The observation \cite[Remark 3.2.3c]{Fulton84} that is key for us is that the Chern roots of $\schurfunc{(1^k)}(\mathcal{E})=\Exterior^{k}\mathcal{E}$ for any positive integer $1\leq k\leq r$ are given by
\[
\left\lbrace\sum_{i\in S}\alpha_i: S\subseteq [r], |S|=k\right\rbrace.
\]
This should remind the reader of the construction of the alphabet $X^{(k)}$  from $X$.


From this point onwards, fix a complex vector bundle $\mathcal{E}$ of
rank $n$.
Given a positive integer $k$, let $\delta_k$ be the partition of
staircase shape $(k,k-1,\ldots,1)$.  We have the following influential 
theorem due to Lascoux.
\begin{theorem}\cite{Lascoux78}\label{thm:Lascoux}
The total Chern class of $\Exterior^{2}\mathcal{E}$ and
$\mathrm{Sym}^{2}\mathcal{E}$ is Schur-positive in terms of the Chern
roots $x_1,\ldots, x_n$ of $\mathcal{E}$.  Specifically, there exist 
integers $d_{\lambda,\mu}\geq 0$  for $\mu \subseteq \lambda$
such that
\begin{align*}
&c(\Exterior^2\mathcal{E})=\prod_{1\leq i <j\leq n}(1+x_i+x_j)=2^{-\binom{n}{2}}\sum_{\mu\subseteq \delta_{n-1}}d_{\delta_{n-1},\mu}2^{|\mu|}s_{\mu}(X),\\
&c(\mathrm{Sym}^2\mathcal{E})=\prod_{1\leq i \leq j\leq n}(1+x_i+x_j)=2^{-\binom{n}{2}}\sum_{\mu\subseteq \delta_n}d_{\delta_n,\mu}2^{|\mu|}s_{\mu}(X).
\end{align*}
\end{theorem}
The $d_{\lambda,\mu}$ appearing in Theorem~\ref{thm:Lascoux} are defined as follows: pad  $\lambda$ and $\mu$ with $0$s so that the resulting sequences have length $n$ each. Say we obtain $(\lambda_1,\ldots,\lambda_n)$ and $(\mu_1,\ldots, \mu_n)$ from $\lambda$ and $\mu$ respectively.
Then, assuming $\mu\subseteq \lambda$, 
\[
d_{\lambda,\mu}=\det\left( \binom{\lambda_i+n-i}{\mu_j+n-j}\right)_{1\leq i,j\leq n}.
\]
Determinants such as the one above are called \emph{binomial
  determinants}.  It is not immediate that these determinants are
positive.  Lascoux \cite{Lascoux78} appeals to geometric
considerations to establish positivity.  Establishing the positivity
combinatorially is the primary motivation of the seminal work of
Gessel-Viennot \cite{Gessel-Viennot85} who identify $d_{\lambda,\mu}$
as counting certain non-intersecting lattice paths in the
plane. 
This combinatorial interpretation implies that the coefficients in the
expansion in Theorem~\ref{thm:Lascoux} are positive rational numbers.
To prove integrality, observe that the products yielding
$c(\Exterior^{2}\mathcal{E})$ and $c(\mathrm{Sym}^2\mathcal{E})$
expand integrally in the basis of monomial symmetric polynomials.  The
inverse of the Kostka matrix, whose entries are integral, allows us to
obtain an integral expansion in terms of Schur polynomials.  To the
best of our knowledge, there is no known combinatorial proof
establishing that $2^{\binom{n}{2}}$ divides
$d_{\delta_{n-1},\mu}2^{|\mu|}$.

Given partitions $\lambda$ and $\mu$ not necessarily comparable by containment, denote by $s_{\lambda}(\schurfunc{\mu}(\mathcal{E}))$ the Schur polynomial $s_{\lambda}$ evaluated at the alphabet comprising the Chern roots of $\schurfunc{\mu}(\mathcal{E})$.
\begin{example}\label{ex:Pragacz's function}
Let $n=3$ and $X=\{x_1,x_2,x_3\}$ consist of the Chern roots of some vector bundle $\mathcal{E}$ of rank $3$. Then we have  
\begin{align*}
s_{(21)}(\schurfunc{(1^2)}(\mathcal{E}))=s_{(21)}(x_1+x_2,x_1+x_3,x_2+x_3)
= 2s_{(3)}(X)+5s_{(21)}(X)+4s_{(111)}(X).
\end{align*}
\end{example}

\begin{remark}\label{rem:Not plethysm}
The reader should not confuse the earlier operation of substituting the alphabet corresponding to the Chern roots of $\schurfunc{\mu}(\mathcal{E})$ into $s_{\lambda}$ for plethysm, which corresponds to taking the character of $\schurfunc{\lambda}(\schurfunc{\mu}(\mathcal{E}))$.
\end{remark}

We recall a theorem due to Pragacz which generalizes Lascoux's result above.
The gist of the statement is also present in \cite[Page
34]{Pragacz04}.
\begin{theorem}\label{thm:Pragacz} \cite[Corollary 7.2]{Pragacz96}
 Let $\mathcal{E}_1, \ldots, \mathcal{E}_k$ be vector bundles, and let $Y_1,\ldots, Y_k$ be the alphabets consisting of their Chern roots respectively. For partitions $\lambda, \mu^{(1)},\ldots, \mu^{(k)}$, there exists nonnegative integers $c_{(\nu^{(1)},\ldots,\nu^{(k)})}^{\lambda, (\mu^{(1)},\ldots,\mu^{(k)})}$ such that 
 \[
 s_{\lambda}(\schurfunc{\mu^{(1)}}(\mathcal{E}_1)\otimes \cdots \otimes \schurfunc{\mu^{(k)}}(\mathcal{E}_k))=\sum_{\nu_1,\ldots, \nu_k} c_{(\nu^{(1)},\ldots,\nu^{(k)})}^{\lambda,(\mu^{(1)},\ldots,\mu^{(k)})} s_{\nu_1}(Y_1)\cdots s_{\nu_k}(Y_k).
 \]
\end{theorem}
\noindent Pragacz's proof of Theorem~\ref{thm:Pragacz} relies on deep work of
Fulton-Lazarsfeld \cite{Fulton-Lazarsfeld83} in the context of
numerical positivity. The Hard Lefschetz theorem is a key component in
the aforementioned work.

\section{Schur-positivity using Chern roots}\label{s:main}

In this section, we establish the Schur-positivity of $B_{n,k}(X)$ and
$B_n(X)$.    Additional consequences are then described below.

\begin{proof}[Proof of Theorem~\ref{thm:Billey-Billera}]
Let $\mathcal{E}$ be a complex vector bundle of rank $n$.  Observe
that $c(\Exterior^{2}\mathcal{E})=\sum_p e_p(X^{(2)})$ provided $X$ is
the alphabet of Chern roots of $\mathcal{E}$.  On comparing
homogeneous summands on the right hand side of the preceding equality
with those in Theorem~\ref{thm:Lascoux}, we see that Lascoux's result
yields the Schur-positivity of $e_{p}(X^{(2)})$ for each $p\geq 0$. In
particular, $B_{n,2}(X)$ is Schur-positive.  To establish positivity in the
general case, we follow the route laid out by Lascoux.

For a positive integer $k$, recall from Section~\ref{s:notation} that
the Chern roots of
$\schurfunc{(1^k)}(\mathcal{E})=\Exterior^{k}\mathcal{E}$ are given by
the elements in the alphabet $X^{(k)}$.  We have
\[
c(\Exterior^{k}\mathcal{E})=\prod_{S\subseteq [n], |S|=k} \left(1+\sum_{i\in S}x_i\right)=\sum_{p\geq 0}e_{p}(X^{(k)}).
\]
We will show that each $e_{p}(X^{(k)})$ is Schur-positive in terms of the Chern roots of $\mathcal{E}$.
  
From Theorem~\ref{thm:Pragacz}, we infer
that  the structure coefficients $c^{\lambda,\mu}_{\nu}$  in the
following expansion are all nonnegative, 
\begin{align}\label{eqn: Pragacz}
s_{\lambda}(\mathcal{S}^{\mu}(\mathcal{E}))=\displaystyle\sum_{\nu}c^{\lambda,\mu}_{\nu}s_{\nu}(X).
\end{align}
In the case where $\mu=(1^k)$ for some positive integer $k$ and
$\lambda=(1^p)$ for some nonnegative integer
$0\leq p\leq \binom{n}{k}$, the left hand side of \eqref{eqn: Pragacz}
equals
$e_p(X^{(k)})$. This establishes the Schur-positivity of
$e_p(X^{(k)})$. 
The Schur-positivity of $B_{n,k}(X)$ is the special case $p=\binom{n}{k}$.
Finally,  the Schur-positivity of  $B_n(X)=\prod_{k=1}^{n}B_{n,k}(X)$ follows from the 
Littlewood-Richardson rule, which is an explicit positive combinatorial
rule to multiply Schur polynomials.
\end{proof}



\begin{remark}
One can show $B_{n,2}(X)=e_{\binom{n}{2}}(X^{(2)})=s_{\delta_{n-1}}(X)$ as in \cite[\S 3, Example 7]{Macdonald95}.
However, we do not know a combinatorial proof for the positivity of $e_p(X^{(2)})$ for $1<p<\binom{n}{2}$, or of $e_{p}(X^{(k)})$ for higher $k$ in general.
\end{remark}  


\subsection{The $(n,n-1)$-Boolean product polynomial: A special case}
Let $q$ be an indeterminate and consider the $q$-deformation of $B_{n,n-1}(X)$ defined as 
\begin{align}\label{eqn:q-deformation}
B_{n,n-1}(X;q)\coloneqq\prod_{i=1}^{n}(h_1(X)+qx_i).
\end{align}
To motivate the above deformation, we consider the special cases $q=0$ and $q=-1$.
Observe that $B_{n,n-1}(X;0)=h_{(1^n)}(X)$.
Let $V$ be an $n$-dimensional vector space over $\mathbb{C}$. If $\rho$ is the polynomial representation of $\mathrm{GL}_n(\mathbb{C})$ obtained by its action on $V^{\otimes n}$, then its character $\mathsf{Ch}(\rho)$ is equal to $h_{(1^n)}(X)$.
Furthermore, recall that $h_{(1^n)}(X)$ is also the Frobenius characteristic of the regular representation of $\mathfrak{S}_n$. Thus, in the $q=0$ case, we recover well-known representations.

The case $q=-1$ is more interesting.
 Clearly, $B_{n,n-1}(X;-1)$ equals $B_{n,n-1}(X)$.  On expanding  the product in \eqref{eqn:q-deformation}, we obtain
 \begin{align}
 B_{n,n-1}(X)=\sum_{j=0}^{n}(-1)^{j}e_{j}(X)h_{(1^{n-j})}(X).
 \end{align}
 Comparing the expression on the right hand side with the equality in \cite[Theorem 8.1]{GesselReutenauer93}, we conclude that $B_{n,n-1}(X)$ equals the  symmetric function denoted by $D_n$ therein.
 The symmetric function $D_n$ was introduced by Desarmenien and Wachs \cite{DesarmenienWachs88} in the context of descents sets of derangements. 
Given the expansion of $D_n$ in the basis of fundamental quasisymmetric functions, we obtain the following result. 
\begin{theorem}
   For $n\geq 2$, we have the Schur-positive expansion $\displaystyle 
   B_{n,n-1}(X)=\sum_{\lambda\vdash n}a_\lambda s_{\lambda}(X) $ where
   $a_\lambda$ is the number of $T\in SYT(\lambda)$ with smallest
   ascent given by an even number.
 \end{theorem}

Now consider the case where $q$ is a positive integer. We have
 \begin{align}\label{eqn:generic q}
 B_{n,n-1}(X;q)=\sum_{j=0}^{n}q^{j}e_{j}(X)h_{(1^{n-j})}(X).
 \end{align}
 From \eqref{eqn:generic q}  it is clear, for instance by the Pieri rule, that $B_{n,n-1}(X;q)$ is Schur-positive. 
 We briefly remark on how to construct (ungraded) $\mathfrak{S}_n$-modules whose Frobenius characteristic is $B_{n,n-1}(X;q)$.  
 Let $\mathbf{1}$ denote the trivial character of the Young subgroup $\mathfrak{S}_1^{j}\times \mathfrak{S}_{n-j}$ of $\mathfrak{S}_n$.
 Then the Frobenius characteristic of the induced character $\mathbf{1}\!\!\uparrow_{\mathfrak{S}_1^{j}\times \mathfrak{S}_{n-j}}^{\mathfrak{S}_n}$ is equal to $h_{n-j}(X)h_{(1^j)}(X)$.
 On taking a direct sum of $q^j$ copies of this induced character for $0\leq j\leq n$ and subsequently tensoring with the sign representation of $\mathfrak{S}_n$, we obtain a character with Frobenius characteristic $B_{n,n-1}(X;q)$.
 This construction is mildly unsatisfactory and one would ideally want a more `natural' graded representation where $q$ records the grading.

 We conclude with a curious observation when $q=1$. In this case, the dimension of a $\mathbb{C}\mathfrak{S}_n$-module with Frobenius characteristic $B_{n,n-1}(X;1)$ is equal to $\displaystyle\sum_{k=0}^n \frac{n!}{k!}$.
 By \cite[Theorem 10.4]{ArdilaRinconWilliams16}, this is also the number of positroids on
 $[n]$. In ongoing work, we are investigating a natural $\mathfrak{S}_n$-action on the distinguished 
 indexing set for positroids given by decorated permutations.
 See \cite[A000522]{oeis} for many further interpretations of this sequence of dimensions.


\subsection{The case of two alphabets}
We investigate the case where we have two different alphabets. 
The result that follows allows for further generalization to the case of multiple alphabets.

Let $X=\{x_1,\ldots, x_n\}$ and $Y= \{y_1,\ldots, y_m\}$ be two distinct alphabets. 
Note that the cardinalities of $X$ and $Y$ can be distinct.
Assume further that $X$ and $Y$ consist of Chern roots of vector bundles $\mathcal{E}$ and $\mathcal{F}$ of ranks $n$ and $m$ respectively.
 Given nonempty subsets $S\subseteq [n]$ and $T\subseteq [m]$, define the subset sums
\begin{align*}
X_S\coloneqq\sum_{i\in S}x_i \text{ and }
Y_T\coloneqq\sum_{i\in T}y_i.
\end{align*}
Fix positive integers $j$ and $k$. Consider the following product that naturally generalizes the $(n,k)$-th Boolean product polynomial.
\begin{align*}
\mathcal{P}_{j,k}(X,Y)\coloneqq\prod_{\substack{S\subseteq [n]\\|S|=j}}\prod_{\substack{T\subseteq [m]\\|T|=k}} (X_S+Y_T).
\end{align*}
This expression is clearly symmetric in the $X$ variables and $Y$ variables. Note further that $\mathcal{P}_{j,k}(X,Y)$ is equal to $e_p(\Exterior^{j}\mathcal{E}\otimes \Exterior^{k}\mathcal{F})$ where $p=\binom{n}{j}\binom{m}{k}$. 
 Therefore, by invoking Theorem~\ref{thm:Pragacz} again, we obtain the following extension.
 \begin{theorem}\label{cor:multiple alphabets}
The bivariate polynomial $\mathcal{P}_{j,k}(X,Y)$ is Schur-positive. That is, there exist nonnegative integers $a_{\lambda\mu}$ such that 
\begin{align*}\label{eqn: dual Cauchy generalization}
\mathcal{P}_{j,k}(X,Y)
=\sum_{\lambda,\mu}a_{\lambda\mu}s_{\lambda}(X)s_{\mu}(Y).
\end{align*}
\end{theorem}
\noindent Observe that the Schur-positivity in Theorem~\ref{cor:multiple alphabets} subsumes that in the statement of Theorem~\ref{thm:Billey-Billera} if we pick exactly one of $j$ or $k$ to equal $0$.

 \section{Further remarks}\label{s:remarks}
\begin{enumerate}
\item Recall that part of motivation for studying Boolean product polynomials was to understand the matroid $M_n$ spanned by nonzero $n$-vectors with components $0$ or $1$.  We should note here that understanding the real linear algebra of these vectors can go much deeper than knowledge of the matroid $M_{n}$, which effectively only needs to know for each $n \times n$ matrix $A$ of $0's$ and $1's$ whether $\det A$ is zero or not.  For example, to ``know'' the \emph{arithmetic matroid}  of all 0-1 vectors, one needs to know, in addition, the absolute value $|\det A|$ of each such matrix.  Now to really know the possible determinants of all 0-1 matrices would include the solution of the problem of Hadamard, \emph{i.e.}, whether there is an $n \times n$ Hadamard matrix whenever $n=4k$.  The reason for this is that for each 0-1 $n \times n$ matrix $A$, $\det A \le (n+1)^{(n+1)/2}/2^{n}$ with equality if and only if there is a Hadamard matrix of order $n+1$.

\item From the work of Gessel-Reutenauer \cite[Theorem 3.6]{GesselReutenauer93}, it  follows that $B_{n,n-1}(X)$ is also related to representations of the free Lie algebra. Let $V$ be an $n$-dimensional vector space over $\mathbb{C}$. 
Then we have 
\[
B_{n,n-1}(X)=\sum_{\substack{\lambda\vdash n\\\lambda \text{ does not have parts equaling } 1}}\mathsf{Ch}(\mathrm{Lie}_{\lambda}(V)),
\]
where $\mathrm{Lie}_{\lambda}(V)$ are certain $\mathrm{GL}(V)$-modules known as \emph{higher Lie modules} \cite{Reutenauer93}.
This suggests that there might be a link between Boolean product polynomials and representations of the free Lie algebras, and we intend to explore this in the future.

\item Observe that
$
\mathcal{P}_{1,1}(X,Y)=\prod_{i=1}^{n}\prod_{j=1}^{m}(x_i+y_j)
$
and by the dual Cauchy identity, this product is well-known to have a nice Schur function expansion.
Thus, a natural question is to find an appropriate analogue to the Robinson-Schensted insertion algorithm that allows us to establish the Schur-positivity in Corollary~\ref{cor:multiple alphabets} combinatorially.

\end{enumerate}



\printbibliography


\end{document}